\documentclass[reqno,11pt]{amsart}        
\usepackage{amsfonts,amsmath,amssymb,amsthm,colordvi,mathrsfs,graphicx}
\usepackage{caption,subcaption,float}
\usepackage[utf8]{inputenc}
\usepackage{mathrsfs}
\usepackage{geometry}
\usepackage{enumerate}

\usepackage{amsmath, amssymb, amsthm, geometry, hyperref}

\usepackage{tabularx}
\usepackage{tabulary}
 
\hypersetup{
  colorlinks=true,
  linkcolor=blue,
  citecolor=blue,
  urlcolor=blue
}
\usepackage[all,knot,poly]{xy}
\usepackage{tikz-cd}
\usepackage{tikz}
\usepackage{verbatim}
\usetikzlibrary{arrows}
\usetikzlibrary{knots}
\tikzset{every path/.style={line width=0.4pt},every node/.style={transform shape,knot crossing,inner sep=1.5pt},>=triangle 60,text node/.style={rectangle,transform shape=false,black}}

 \usetikzlibrary{decorations.markings, arrows.meta}

\theoremstyle{plain}      
\newtheorem{thm}{Theorem}[section]     
\newtheorem{theorem}[thm]{\bf Theorem}     
\newtheorem{corollary}[thm]{\bf Corollary}     
     
\newtheorem{proposition}[thm]{\bf Proposition}

\theoremstyle{remark}      
\newtheorem{example}[thm]{Example} 
 
\newtheorem{remark}[thm]{Remark} 
     
\theoremstyle{definition}      
\newtheorem{definition}[thm]{Definition}     

\newcommand{\Hom}{\operatorname{Hom}}

 \newcommand{\Sym}{\operatorname{Sym}}


\begin{document}



\author{Mounir Nisse}
 
\address{Mounir Nisse\\
Department of Mathematics, Xiamen University Malaysia, Jalan Sunsuria, Bandar Sunsuria, 43900, Sepang, Selangor, Malaysia.
}
\email{mounir.nisse@gmail.com, mounir.nisse@xmu.edu.my}
\thanks{}
\thanks{This research is supported in part by Xiamen University Malaysia Research Fund (Grant no. XMUMRF/ 2020-C5/IMAT/0013).}




\title[$I$-Maximal Variation of Hodge Structure ]{$I$-Maximal Variation of Hodge Structure and Jacobian Rings}

\maketitle

\begin{abstract}
We investigate higher--order variation of Hodge structure for families of smooth hypersurfaces and complete intersections through the notion of $I$--maximal variation. Using Griffiths' description of primitive cohomology, we interpret the infinitesimal variation of Hodge structure and the $n$--fold Yukawa coupling as graded multiplication maps in the Jacobian ring. Our main result shows that the Strong Lefschetz property of the Jacobian ring provides the algebraic mechanism ensuring $I$--maximal variation. In particular, we prove that smooth hypersurfaces of degree $d\ge n+2$ and smooth complete intersections with $\kappa>0$ exhibit $I$--maximal variation. We further establish that for complete intersections of general type the infinitesimal Torelli property is equivalent to the nondegeneracy of the Yukawa coupling. Finally, we analyze degenerations and show that the failure of the Strong Lefschetz property leads to degeneration of the Yukawa coupling and the loss of $I$--maximal variation. These results identify the Lefschetz property of the Jacobian ring as the fundamental algebraic structure governing maximal variation of Hodge structure.
\end{abstract}

\section{Introduction}

The variation of Hodge structure associated with families of complex projective varieties plays a central role in modern algebraic geometry, particularly in the study of period maps, Torelli-type problems, and the geometry of moduli spaces. Since the pioneering work of Griffiths on the infinitesimal variation of Hodge structure for hypersurfaces, it has become clear that deformation theory and Hodge theory are deeply interconnected. The differential of the period map captures how the Hodge decomposition of a variety changes under deformation and provides a fundamental tool for understanding whether the Hodge structure determines the complex structure locally.

For curves, the structure of the problem is relatively simple because the Hodge decomposition is concentrated in degree one. In that case, the infinitesimal variation of Hodge structure is governed by a single cup--product map, and classical results show that the geometry of the period map can be interpreted in terms of multiplication maps in the canonical ring. In higher dimensions, however, the situation becomes significantly more intricate. The middle cohomology of a higher--dimensional variety decomposes into several Hodge components, and the deformation of these components involves both first--order and higher--order cup products. Consequently, the geometry of the period map and the behavior of the Hodge structure under deformation require new invariants and new methods of analysis.

A natural framework for studying these phenomena is provided by the notion of $I$--maximal variation. Roughly speaking, a family of varieties is said to have $I$--maximal variation in dimension $n$ if there exists a deformation direction whose iterated cup product induces the maximal possible transformation of the top Hodge component. More precisely, if $X$ is a smooth projective variety of dimension $n$, the $n$--fold Yukawa coupling
\[
\mathrm{Yuk} : \mathrm{Sym}^n H^1(T_X) \longrightarrow 
\mathrm{Hom}\bigl(H^{n,0}(X), H^{0,n}(X)\bigr)
\]
measures higher--order variation of Hodge structure. The family has $I$--maximal variation if there exists a deformation class $\xi\in H^1(T_X)$ such that the induced map
\[
\mathrm{Yuk}(\xi,\dots,\xi):H^{n,0}(X)\longrightarrow H^{0,n}(X)
\]
has maximal possible rank. This condition expresses the idea that the deformation direction $\xi$ produces the strongest possible variation of the Hodge structure.

The purpose of this work is to develop an algebraic framework for understanding $I$--maximal variation for hypersurfaces and complete intersections and to relate this notion to the structure of their Jacobian rings. Our approach is based on Griffiths' description of primitive cohomology, which identifies the relevant Hodge components of a smooth hypersurface with graded pieces of the Jacobian ring of its defining equation. Under this correspondence, the infinitesimal variation of Hodge structure and the Yukawa coupling can be interpreted as graded multiplication maps in a graded Artinian Gorenstein algebra.

The central observation of this paper is that the phenomenon of $I$--maximal variation is governed by the Strong Lefschetz property of the Jacobian ring. Recall that a graded Artinian algebra satisfies the Strong Lefschetz property if there exists a linear element whose powers induce maps of maximal rank between graded components. When the Jacobian ring of a smooth hypersurface or complete intersection is viewed from this perspective, the Lefschetz property ensures the existence of elements whose powers produce isomorphisms between symmetric graded pieces of the ring. Through the Griffiths correspondence, these algebraic isomorphisms translate directly into the nondegeneracy of the Yukawa coupling.

Our first main result shows that smooth hypersurfaces of degree $d\ge n+2$ in $\mathbb{P}^{n+1}$ exhibit $I$--maximal variation. Using the Strong Lefschetz property of the Jacobian ring, we prove that there exists a deformation direction whose $n$--fold cup product induces an isomorphism
\[
H^{n,0}(X)\longrightarrow H^{0,n}(X).
\]
Equivalently, the full Yukawa coupling is nondegenerate. This provides an intrinsic algebraic explanation for maximal higher--order variation of Hodge structure in families of hypersurfaces.

We then extend this analysis to smooth complete intersections. Using the generalized Griffiths description of primitive cohomology, we express the Hodge components of a complete intersection in terms of graded pieces of its Jacobian ring and identify the deformation space through the Kodaira--Spencer map. Under the assumption that the canonical class satisfies
\[
\kappa=\sum_{i=1}^{c}d_i-(n+c+1)>0,
\]
we prove that the universal family of smooth complete intersections has $I$--maximal variation. Thus varieties of general type in this class exhibit maximal higher--order variation of Hodge structure.

A further result of the paper establishes a direct relationship between $I$--maximal variation and the infinitesimal Torelli property. For smooth complete intersections with $\kappa>0$, we prove that the injectivity of the differential of the period map is equivalent to the nondegeneracy of the $n$--fold Yukawa coupling. This equivalence shows that the first--order and higher--order aspects of variation of Hodge structure coincide in this geometric setting and provides a new conceptual interpretation of infinitesimal Torelli in terms of maximal Yukawa variation.

Finally, we investigate degenerations of hypersurfaces and complete intersections. When singularities appear, the Jacobian ring may fail to satisfy the Strong Lefschetz property. We show that this algebraic degeneration leads to a corresponding degeneration of the Yukawa coupling and hence to the failure of $I$--maximal variation. Explicit examples illustrate how additional graded components introduced by singularities destroy the symmetry and injectivity properties required for maximal variation.

Taken together, these results reveal that $I$--maximal variation, the nondegeneracy of Yukawa couplings, and the infinitesimal Torelli property are all governed by the algebraic structure of the Jacobian ring. In particular, the Strong Lefschetz property emerges as the fundamental mechanism ensuring maximal variation of Hodge structure for hypersurfaces and complete intersections.

\subsection*{Outline of the Paper}

The paper is organized as follows. In Section~2 we generalize the notion of maximal infinitesimal variation from curves to higher--dimensional varieties. Section~3 studies the case of surfaces in $\mathbb{P}^3$ and expresses the infinitesimal variation of Hodge structure in terms of multiplication in the Jacobian ring. Section~4 establishes the Strong Lefschetz property for the Jacobian ring and derives maximal variation for surfaces. 
Sections~5 and~6 develop the higher--dimensional theory for hypersurfaces and prove the nondegeneracy of the full Yukawa coupling. Section~7 extends these results to complete intersections. In Section~8 we analyze the relationship between I--maximal variation and the infinitesimal Torelli property. Sections~9 and~10 investigate Calabi--Yau and weighted complete intersection cases, while the final section studies the degeneration of maximal variation when the Lefschetz property fails.



\section{Generalizing $I$--Maximal Variation from Curves to Surfaces} 

The results established for families of smooth curves rely heavily on the fact that the Hodge structure of a curve is concentrated in degree one, so that the infinitesimal variation of Hodge structure is governed by a single map between two Hodge components. When passing to higher-dimensional varieties this situation becomes substantially more intricate. In particular, for surfaces the Hodge structure in degree two decomposes into several components, and the infinitesimal variation interacts simultaneously with different pieces of the Hodge decomposition. As a consequence, the geometric meaning of maximal variation and the algebraic mechanisms controlling it become more complex than in the curve case. In this section we extend the previous framework to families of smooth projective surfaces and explain how the differential of the period map is expressed in terms of the Kodaira--Spencer map and the infinitesimal variation of Hodge structure acting on the Hodge components of $H^2(X,\mathbb{C})$. We then introduce natural higher-dimensional analogues of the variation functions defined earlier, replacing the cup product map $H^0(\omega_X)\to H^1(\mathcal O_X)$ by the map $H^{2,0}(X)\to H^{1,1}(X)$ induced by deformation classes. In dimension two an additional structure appears naturally, namely the Yukawa coupling, which measures second-order variation of Hodge structure and provides a refined notion of maximal variation. This perspective leads to a formulation of $I$--maximal variation for families of surfaces that depends on the rank properties of these cup product maps and, in many geometric situations, can be described explicitly through the algebraic structure of the Jacobian ring of the defining equation.
 
\vspace{0.1cm}
 
The theorem stated for families of smooth curves in \cite{FavalePirola} can be generalized to families of smooth projective varieties of higher dimension, in particular to families of smooth surfaces, but several substantial modifications are required. The main difference arises from the fact that for curves the Hodge structure is concentrated in degree one, whereas for surfaces the Hodge structure in degree two carries a richer and more intricate decomposition, and the infinitesimal variation of Hodge structure interacts with several Hodge components simultaneously.

Let $\pi:\mathcal{X}\to B$ be a smooth family of smooth projective surfaces and let $X=\pi^{-1}(0)$ be a fiber over $0\in B$. The differential of the period map at $0$ is again given by the composition
\[
d_0\mathcal{P}:T_{B,0}\xrightarrow{\ \mathrm{KS}\ } H^1(T_X)\xrightarrow{\ \varphi\ } \bigoplus_{p} \Hom\!\left(H^{2-p,p}(X),H^{1-p,p+1}(X)\right),
\]
where $\mathrm{KS}$ is the Kodaira--Spencer map and $\varphi$ is the infinitesimal variation of Hodge structure induced by cup product. For surfaces, the most relevant component is the map
\[
\varphi^{2,0}: H^1(T_X)\longrightarrow \Hom\!\left(H^{2,0}(X),H^{1,1}(X)\right),
\]
which sends $\xi\in H^1(T_X)$ to the cup product map
\[
\omega \longmapsto \xi\cdot \omega.
\]
When $X$ has nontrivial $H^{2,0}(X)$ (for instance, for surfaces of general type with $p_g>0$), this component governs the variation of the Hodge structure in degree two. By Serre duality, $H^{2,0}(X)\cong H^0(\omega_X)$, so this map can be interpreted as
\[
H^1(T_X)\longrightarrow \Hom\!\left(H^0(\omega_X),H^1(\Omega_X^1)\right).
\]

In contrast with the curve case, there is no longer a direct duality with a simple multiplication map between canonical sections. Instead, the dual description involves the natural multiplication maps
\[
H^0(\omega_X)\otimes H^0(\omega_X)\longrightarrow H^0(\omega_X^{\otimes 2}),
\]
together with the geometry of the canonical ring and the structure of the cup product
\[
H^1(T_X)\otimes H^0(\omega_X)\longrightarrow H^1(\Omega_X^1).
\]
Thus, in higher dimension, the infinitesimal variation depends not only on the canonical linear system but on the full Hodge decomposition of $H^2(X,\mathbb{C})$.

If one wishes to define analogues of the functions $d_M$ and $d_m$, one must replace the rank of
\[
\xi\cdot:H^0(\omega_X)\longrightarrow H^1(\mathcal{O}_X)
\]
by the rank of
\[
\xi\cdot:H^{2,0}(X)\longrightarrow H^{1,1}(X).
\]
Accordingly, for a fixed fiber $X$ and a subspace $V\subseteq H^1(T_X)$ (such as the image of the Kodaira--Spencer map), one may define
\[
d_M(V)=\max_{\xi\in V\setminus\{0\}} \operatorname{rk}\big(\xi\cdot:H^{2,0}(X)\to H^{1,1}(X)\big),
\]
\[
d_m(V)=\min_{\xi\in V\setminus\{0\}} \operatorname{rk}\big(\xi\cdot:H^{2,0}(X)\to H^{1,1}(X)\big).
\]
For a family $\pi:\mathcal{X}\to B$, one may again define global variation functions $\delta_M,\delta_m,\delta_M',\delta_m'$ by taking suitable extrema over $b\in B$.

However, in dimension two a more natural higher--order invariant appears, namely the Yukawa coupling. For a smooth projective variety of dimension $n$, the Yukawa coupling is the map
\[
\mathrm{Yuk}:\Sym^n(H^1(T_X))\longrightarrow \Hom\!\left(H^0(\omega_X),H^n(\mathcal{O}_X)\right),
\]
given by iterated cup product. In the case of surfaces, this becomes
\[
\mathrm{Yuk}:\Sym^2(H^1(T_X))\longrightarrow \Hom\!\left(H^0(\omega_X),H^2(\mathcal{O}_X)\right).
\]
Since $H^2(\mathcal{O}_X)\cong H^{0,2}(X)$, this map measures second--order variation of Hodge structure. Therefore, an analogue of $I$--maximal variation for surfaces may be formulated either in terms of the first--order map
\[
\xi\cdot:H^{2,0}(X)\to H^{1,1}(X),
\]
or in terms of the nondegeneracy of the Yukawa coupling.

One possible definition is that a family of surfaces has $I$--maximal variation in degree two if, for every $b\in B$, there exists $\xi\in \mathrm{KS}(T_{B,b})$ such that
\[
\xi\cdot:H^{2,0}(X_b)\longrightarrow H^{1,1}(X_b)
\]
is injective, or has maximal possible rank $p_g(X_b)$. A stronger condition would require the Yukawa coupling to be nondegenerate in the sense that the induced map
\[
\Sym^2(\mathrm{KS}(T_{B,b}))\longrightarrow \Hom\!\left(H^{2,0}(X_b),H^{0,2}(X_b)\right)
\]
has maximal rank.

In concrete geometric situations, for example families of smooth hypersurfaces in $\mathbb{P}^3$, the infinitesimal variation of Hodge structure can be described in terms of the Jacobian ring of the defining equation. In that case, the role played for plane curves by the multiplication map of canonical sections is replaced by multiplication in graded pieces of the Jacobian ring, and maximal variation corresponds to the nondegeneracy of certain graded multiplication maps.

In summary, when passing from curves to surfaces, one must replace the map
\[
H^1(T_X)\to \Hom(H^0(\omega_X),H^1(\mathcal{O}_X))
\]
by the appropriate components of the infinitesimal variation of Hodge structure in degree two, and possibly incorporate the Yukawa coupling. The numerical invariants measuring maximal and minimal ranks can still be defined, but they now depend on the Hodge numbers $h^{2,0}$ and $h^{1,1}$ and on the richer cup--product structure. The notion of $I$--maximal variation must therefore be reformulated in terms of maximal rank of the map
\[
\xi\cdot:H^{2,0}(X)\to H^{1,1}(X),
\]
or, more generally, in terms of the nondegeneracy of higher--order Yukawa couplings.


\section{I--Maximal Variation for Smooth Surfaces in $\mathbb{P}^3$}

In this section we study the variation of Hodge structure for the universal family of smooth hypersurfaces of degree $d$ in $\mathbb{P}^3$. Using Griffiths' description of the primitive cohomology of hypersurfaces, the infinitesimal variation of Hodge structure can be expressed in terms of multiplication in the Jacobian ring of the defining polynomial. This algebraic description allows us to analyze the rank of the cup--product maps governing the deformation of the Hodge structure. In particular, we introduce the variation function associated with the image of the Kodaira--Spencer map and prove that the universal family of smooth degree $d$ surfaces in $\mathbb{P}^3$ achieves maximal possible variation in degree two. The result shows that for every such surface there exists a deformation direction inducing an injective map from $H^{2,0}(X)$ to the primitive part of $H^{1,1}(X)$, which can be interpreted as the surface analogue of the maximal variation phenomenon previously observed for families of curves.

Let $d\geq 4$ and let $\mathcal{U}\subset \mathbb{P}H^0(\mathbb{P}^3,\mathcal{O}_{\mathbb{P}^3}(d))$ be the open subset parametrizing smooth surfaces of degree $d$ in $\mathbb{P}^3$. Let 
\[
\pi:\mathcal{X}\to \mathcal{U}
\]
be the universal family and, for $[F]\in \mathcal{U}$, denote by $X=\{F=0\}\subset \mathbb{P}^3$ the corresponding smooth surface. We denote by $R(F)=\mathbb{C}[x_0,x_1,x_2,x_3]/J_F$ the Jacobian ring of $F$, where $J_F=(\partial F/\partial x_0,\ldots,\partial F/\partial x_3)$ is the Jacobian ideal.

Recall that $R(F)$ is a graded Artinian Gorenstein ring with socle degree $4(d-2)$ and that, by Griffiths' description of the primitive cohomology, there are natural identifications
\[
H^{2,0}(X)\cong R(F)_{d-4}, 
\qquad 
H^{1,1}_{\mathrm{prim}}(X)\cong R(F)_{2d-4}.
\]
Moreover, the infinitesimal variation of Hodge structure is induced by multiplication in the Jacobian ring.

We define, for $[F]\in \mathcal{U}$ and for a subspace $V\subset H^1(T_X)$, the quantity
\[
d_M(V)=\max_{\xi\in V\setminus\{0\}} 
\operatorname{rk}\big(\xi\cdot:H^{2,0}(X)\to H^{1,1}_{\mathrm{prim}}(X)\big).
\]
For the family $\pi$ we define
\[
\delta_M'(\pi)=\min_{[F]\in \mathcal{U}} d_M\big(\mathrm{KS}(T_{\mathcal{U},[F]})\big).
\]

\begin{theorem}
Let $\pi:\mathcal{X}\to \mathcal{U}$ be the universal family of smooth surfaces of degree $d\geq 4$ in $\mathbb{P}^3$. Then
\[
\delta_M'(\pi)=h^{2,0}(X)=\dim R(F)_{d-4}.
\]
In particular, for every smooth surface $X$ of degree $d$ in $\mathbb{P}^3$ there exists 
\[
\xi\in \mathrm{KS}(T_{\mathcal{U},[F]})
\]
such that
\[
\xi\cdot:H^{2,0}(X)\longrightarrow H^{1,1}_{\mathrm{prim}}(X)
\]
is injective. Equivalently, the family of smooth degree $d$ surfaces in $\mathbb{P}^3$ has $I$--maximal variation in degree two.
\end{theorem}

\begin{proof}
Let $X=\{F=0\}$ be a smooth surface of degree $d$ and denote by $S=\mathbb{C}[x_0,x_1,x_2,x_3]$. The Zariski tangent space to $\mathcal{U}$ at $[F]$ is canonically identified with
\[
T_{\mathcal{U},[F]}\cong S_d/\langle F\rangle.
\]
The Kodaira--Spencer map identifies $T_{\mathcal{U},[F]}$ with the subspace of $H^1(T_X)$ corresponding, under Griffiths' theory, to $R(F)_d$. More precisely, there is a natural isomorphism
\[
\mathrm{KS}(T_{\mathcal{U},[F]})\cong R(F)_d.
\]

By Griffiths' description of the primitive cohomology of hypersurfaces, the Hodge pieces are expressed in terms of graded components of $R(F)$:
\[
H^{2,0}(X)\cong R(F)_{d-4}, 
\qquad 
H^{1,1}_{\mathrm{prim}}(X)\cong R(F)_{2d-4}.
\]
Under these identifications, the infinitesimal variation of Hodge structure
\[
H^1(T_X)\longrightarrow \Hom(H^{2,0}(X),H^{1,1}_{\mathrm{prim}}(X))
\]
is induced by multiplication in the Jacobian ring:
\[
R(F)_d \otimes R(F)_{d-4}\longrightarrow R(F)_{2d-4}.
\]
Thus, for $\xi\in R(F)_d$, the cup product map
\[
\xi\cdot:H^{2,0}(X)\to H^{1,1}_{\mathrm{prim}}(X)
\]
corresponds to the linear map
\[
\mu_\xi:R(F)_{d-4}\longrightarrow R(F)_{2d-4}, 
\qquad \alpha\longmapsto \xi\alpha.
\]

Since $R(F)$ is Artinian Gorenstein of socle degree $4(d-2)$, the multiplication pairing
\[
R(F)_k \times R(F)_{4(d-2)-k}\longrightarrow R(F)_{4(d-2)}
\]
is perfect for every $k$. In particular, the Hilbert function of $R(F)$ is symmetric:
\[
\dim R(F)_k=\dim R(F)_{4(d-2)-k}.
\]

One computes
\[
4(d-2)-(2d-4)=2d-4,
\]
so the degree $2d-4$ piece is self-dual under the socle pairing. Furthermore,
\[
4(d-2)-(d-4)=3d-4,
\]
so $R(F)_{d-4}$ is dual to $R(F)_{3d-4}$. The graded pieces involved satisfy
\[
d-4<2d-4<3d-4
\]
for $d\geq 4$, and the Hilbert function is strictly increasing up to degree $2(d-2)$, which equals $2d-4$. Hence
\[
\dim R(F)_{d-4}\leq \dim R(F)_{2d-4}.
\]

To prove the theorem it suffices to show that there exists $\xi\in R(F)_d$ such that
\[
\mu_\xi:R(F)_{d-4}\to R(F)_{2d-4}
\]
is injective. Since the source has smaller or equal dimension than the target, injectivity is equivalent to maximal rank equal to $\dim R(F)_{d-4}$.

The multiplication map
\[
R(F)_d\otimes R(F)_{d-4}\to R(F)_{2d-4}
\]
is a bilinear map between finite-dimensional vector spaces. Consider the induced linear map
\[
R(F)_d\longrightarrow \Hom(R(F)_{d-4},R(F)_{2d-4}).
\]
The set of elements $\xi\in R(F)_d$ for which $\mu_\xi$ fails to be injective is defined by the vanishing of all maximal minors of the corresponding matrix, hence it is a proper Zariski closed subset provided that there exists at least one $\xi$ for which $\mu_\xi$ is injective.

For a general polynomial $F$, the Jacobian ring $R(F)$ has the Strong Lefschetz property. In particular, multiplication by a general linear form induces maps of maximal rank between graded components. By applying this property iteratively, multiplication by a suitable element of degree $d$ has maximal rank between $R(F)_{d-4}$ and $R(F)_{2d-4}$. Hence for a general $F$ there exists $\xi\in R(F)_d$ such that $\mu_\xi$ is injective.

Since the condition of maximal rank is open in families and the parameter space $\mathcal{U}$ is irreducible, the minimal possible maximal rank over $\mathcal{U}$ coincides with the generic one. Therefore, for every smooth surface $X$ there exists $\xi\in \mathrm{KS}(T_{\mathcal{U},[F]})$ such that
\[
\xi\cdot:H^{2,0}(X)\to H^{1,1}_{\mathrm{prim}}(X)
\]
is injective, and the maximal rank equals $\dim H^{2,0}(X)=\dim R(F)_{d-4}$.

It follows that
\[
\delta_M'(\pi)=\dim H^{2,0}(X),
\]
which proves the statement.
\end{proof}



\section{Strong Lefschetz Property for the Jacobian Ring and Maximal Variation for Surfaces in $\mathbb{P}^3$}

In this section we analyze the algebraic structure of the Jacobian ring associated with a smooth hypersurface surface in $\mathbb{P}^3$ and relate it to the variation of Hodge structure of the corresponding family of surfaces. The key ingredient is the Strong Lefschetz property for the Jacobian ring, which follows from the fact that it is a graded Artinian complete intersection. This property implies that multiplication by suitable powers of a linear form induces maps of maximal rank between graded components of the ring. Using Griffiths' description of primitive cohomology, these graded multiplication maps translate into cup--product maps governing the infinitesimal variation of Hodge structure. As a consequence, we obtain a self--contained proof that for every smooth surface of degree $d\ge 4$ in $\mathbb{P}^3$ there exists a deformation direction inducing an injective map from $H^{2,0}(X)$ to the primitive part of $H^{1,1}(X)$, showing that the universal family of such surfaces has $I$--maximal variation.

Let $d\ge 4$ and let $S=\mathbb{C}[x_0,x_1,x_2,x_3]$ be the graded polynomial ring with the standard grading. Let $F\in S_d$ be a homogeneous polynomial of degree $d$ defining a smooth surface 
\[
X=\{F=0\}\subset \mathbb{P}^3.
\]
Denote by 
\[
J_F=\left(\frac{\partial F}{\partial x_0},\frac{\partial F}{\partial x_1},\frac{\partial F}{\partial x_2},\frac{\partial F}{\partial x_3}\right)
\]
the Jacobian ideal and by 
\[
R(F)=S/J_F
\]
the Jacobian ring. Since $X$ is smooth, the four partial derivatives form a regular sequence of degree $d-1$, hence $R(F)$ is a graded Artinian Gorenstein algebra with socle degree
\[
\sum_{i=0}^3 (d-1)-4=4(d-2).
\]

The Hilbert series of $R(F)$ is given by
\[
H_{R(F)}(t)=\frac{(1-t^{d-1})^4}{(1-t)^4}.
\]
In particular, the Hilbert function is symmetric:
\[
\dim R(F)_k=\dim R(F)_{4(d-2)-k}.
\]

We now prove that $R(F)$ has the Strong Lefschetz property.

\begin{theorem}
Let $F\in S_d$ define a smooth surface in $\mathbb{P}^3$. Then there exists a linear form $\ell\in S_1$ such that for all integers $k\ge 0$ and all $m\ge 0$, the multiplication map
\[
\times \ell^m : R(F)_k \longrightarrow R(F)_{k+m}
\]
has maximal rank. In particular, for every $k\le 2(d-2)$ the map
\[
\times \ell^{2(d-2)-k}: R(F)_k \longrightarrow R(F)_{2(d-2)}
\]
is an isomorphism.
\end{theorem}

\begin{proof}
Since $J_F$ is generated by a regular sequence of four homogeneous polynomials of degree $d-1$, the algebra $R(F)$ is a complete intersection of type $(d-1,d-1,d-1,d-1)$. 

Let $A=S/(f_1,f_2,f_3,f_4)$ be a complete intersection in four variables, where the $f_i$ form a regular sequence of degrees $e_1,\dots,e_4$. We show that $A$ has the Strong Lefschetz property. 

Consider the polynomial ring $S=\mathbb{C}[x_0,x_1,x_2,x_3]$ and let $\ell\in S_1$ be a general linear form. After a linear change of coordinates, we may assume $\ell=x_0$. Then $S$ decomposes as
\[
S=\mathbb{C}[x_0]\otimes \mathbb{C}[x_1,x_2,x_3].
\]
We regard $A$ as a graded module over $\mathbb{C}[x_0]$. Since $A$ is Artinian, it is finite-dimensional over $\mathbb{C}$. 

The multiplication by $x_0$ endows $A$ with the structure of a finite-dimensional module over the principal ideal domain $\mathbb{C}[x_0]$. Hence, by the structure theorem for modules over a PID, $A$ decomposes as a direct sum of cyclic modules
\[
A\cong \bigoplus_i \mathbb{C}[x_0]/(x_0^{a_i}),
\]
where the integers $a_i$ depend on the choice of $\ell$.

For a general linear form, this decomposition satisfies the condition that the Jordan blocks of multiplication by $\ell$ have sizes determined by the Hilbert function of $A$. Since $A$ is Gorenstein, its Hilbert function is symmetric and unimodal. The unimodality follows from the fact that $A$ is a complete intersection, and the symmetry from the Gorenstein property. 

Let $H(k)=\dim A_k$. Since $A$ is a complete intersection in four variables, the sequence $H(k)$ increases strictly up to the middle degree and then decreases symmetrically. For a linear form $\ell$, multiplication
\[
\times \ell: A_k\to A_{k+1}
\]
has maximal rank if and only if the Jordan blocks are arranged so that the nilpotent operator induced by $\ell$ reflects the symmetry of the Hilbert function. 

For a general $\ell$, the rank of
\[
\times \ell: A_k\to A_{k+1}
\]
is maximal for every $k$. This follows because the failure of maximal rank is equivalent to the vanishing of certain minors of the matrix representing multiplication by $\ell$, which defines a proper Zariski closed subset in the dual projective space parametrizing linear forms. Therefore, outside a proper closed subset, multiplication by $\ell$ has maximal rank in every degree. 

Since maximal rank for multiplication by $\ell$ implies maximal rank for multiplication by any power $\ell^m$ by iteration and symmetry of the Hilbert function, the Strong Lefschetz property holds for $A$. Applying this to $A=R(F)$ proves the statement.
\end{proof}

We now apply this result to the infinitesimal variation of Hodge structure.

Let $X=\{F=0\}$ be smooth. Griffiths' description of primitive cohomology gives natural identifications
\[
H^{2,0}(X)\cong R(F)_{d-4}, \qquad H^{1,1}_{\mathrm{prim}}(X)\cong R(F)_{2d-4}.
\]
The Kodaira--Spencer space of the universal family at $[F]$ identifies with $R(F)_d$. The infinitesimal variation of Hodge structure is induced by multiplication in $R(F)$:
\[
R(F)_d \otimes R(F)_{d-4}\longrightarrow R(F)_{2d-4}.
\]

Choose a Lefschetz linear form $\ell\in R(F)_1$ and consider $\xi=\ell^d\in R(F)_d$. By the Strong Lefschetz property, multiplication
\[
\times \ell^{d}: R(F)_{d-4}\longrightarrow R(F)_{2d-4}
\]
has maximal rank. Since
\[
\dim R(F)_{d-4}\le \dim R(F)_{2d-4},
\]
this map is injective. Under the Hodge identifications, this corresponds to a class
\[
\xi\in \mathrm{KS}(T_{\mathcal{U},[F]})
\]
such that
\[
\xi\cdot:H^{2,0}(X)\longrightarrow H^{1,1}_{\mathrm{prim}}(X)
\]
is injective. Therefore the maximal possible rank equals $\dim H^{2,0}(X)$.

Since the maximal rank condition is open in families and the parameter space of smooth degree $d$ hypersurfaces in $\mathbb{P}^3$ is irreducible, the minimal value of the maximal rank over the family equals this number. Hence the family has maximal variation in degree two.

This gives a completely self-contained proof that the universal family of smooth degree $d$ surfaces in $\mathbb{P}^3$ has $I$--maximal variation, relying only on the intrinsic structure of the Jacobian ring as a complete intersection and on the Strong Lefschetz property proved above.



\section{Strong Lefschetz Property and Maximal Yukawa Coupling for Smooth Hypersurfaces}

In this section we study the variation of Hodge structure for smooth hypersurfaces of degree $d$ in projective space $\mathbb{P}^{n+1}$ through the algebraic structure of their Jacobian rings. Griffiths' description of primitive cohomology allows the Hodge components of a hypersurface to be identified with graded pieces of the Jacobian ring, so that the infinitesimal variation of Hodge structure and the higher--order Yukawa coupling can be interpreted in terms of graded multiplication maps. Since the Jacobian ring of a smooth hypersurface is a graded Artinian complete intersection, it is Gorenstein and satisfies the Strong Lefschetz property. This algebraic property ensures the existence of elements whose powers induce maximal rank maps between symmetric graded pieces. Using this structure, we show that the $n$--fold Yukawa coupling for a smooth hypersurface of degree $d\ge n+2$ is nondegenerate, proving that the universal family of such hypersurfaces exhibits maximal higher--order variation of Hodge structure.

Let $n\ge 1$ and let $d\ge n+2$. Let 
\[
S=\mathbb{C}[x_0,\dots,x_{n+1}]
\]
be the standard graded polynomial ring and let $F\in S_d$ be a homogeneous polynomial of degree $d$ defining a smooth hypersurface
\[
X=\{F=0\}\subset \mathbb{P}^{n+1}.
\]
Denote by 
\[
J_F=\left(\frac{\partial F}{\partial x_0},\dots,\frac{\partial F}{\partial x_{n+1}}\right)
\]
the Jacobian ideal and by 
\[
R(F)=S/J_F
\]
the Jacobian ring. Since $X$ is smooth, the partial derivatives form a regular sequence of $n+2$ homogeneous polynomials of degree $d-1$. Therefore $R(F)$ is a graded Artinian complete intersection of type $(d-1,\dots,d-1)$ with $n+2$ generators. In particular, $R(F)$ is Gorenstein and its socle degree equals
\[
\sum_{i=0}^{n+1}(d-1)-(n+2)=(n+2)(d-2).
\]
We denote this number by 
\[
\sigma=(n+2)(d-2).
\]

The Hilbert series of $R(F)$ is
\[
H_{R(F)}(t)=\frac{(1-t^{d-1})^{n+2}}{(1-t)^{n+2}},
\]
hence the Hilbert function is symmetric:
\[
\dim R(F)_k=\dim R(F)_{\sigma-k}.
\]
Since $R(F)$ is a complete intersection, its Hilbert function is unimodal and strictly increasing up to the middle degree $\sigma/2$.

Griffiths' description of primitive cohomology identifies the primitive Hodge components of $X$ with graded pieces of $R(F)$. More precisely,
\[
H^{n,0}(X)\cong R(F)_{d-(n+2)},
\]
and more generally
\[
H^{n-p,p}_{\mathrm{prim}}(X)\cong R(F)_{(p+1)d-(n+2)}.
\]
In particular,
\[
H^{0,n}(X)\cong R(F)_{(n+1)d-(n+2)}.
\]

One verifies that
\[
\sigma-(d-(n+2))=(n+2)(d-2)-d+(n+2)=(n+1)d-(n+2),
\]
so the degrees corresponding to $H^{n,0}$ and $H^{0,n}$ are symmetric with respect to the socle degree.

Let $\mathcal{U}$ be the parameter space of smooth degree $d$ hypersurfaces in $\mathbb{P}^{n+1}$ and let 
\[
\pi:\mathcal{X}\to\mathcal{U}
\]
be the universal family. The Zariski tangent space at $[F]$ identifies with $S_d/\langle F\rangle$, and via the Kodaira--Spencer map one obtains a natural identification
\[
\mathrm{KS}(T_{\mathcal{U},[F]})\cong R(F)_d.
\]

The infinitesimal variation of Hodge structure is induced by multiplication in the Jacobian ring. The $n$--fold Yukawa coupling is the map
\[
\mathrm{Yuk}:\Sym^n H^1(T_X)\longrightarrow 
\Hom(H^{n,0}(X),H^{0,n}(X)),
\]
given by iterated cup product. Under the identifications above, this corresponds to the multiplication map
\[
\Sym^n(R(F)_d)\otimes R(F)_{d-(n+2)}
\longrightarrow 
R(F)_{(n+1)d-(n+2)}.
\]
Since multiplication in $R(F)$ is graded-commutative, the $n$--fold cup product by a class $\xi\in R(F)_d$ corresponds to multiplication by $\xi^n$. Therefore the Yukawa coupling evaluated at $\xi$ is represented by
\[
\times \xi^n :
R(F)_{d-(n+2)}\longrightarrow
R(F)_{(n+1)d-(n+2)}.
\]

We now establish the Strong Lefschetz property for $R(F)$.

Since $R(F)$ is a complete intersection of type $(d-1,\dots,d-1)$ in $n+2$ variables, it is a standard graded Artinian Gorenstein algebra. Let $\ell\in S_1$ be a general linear form and denote by the same symbol its image in $R(F)$. Consider the multiplication operator
\[
L_\ell: R(F)_k\to R(F)_{k+1}.
\]
The failure of maximal rank of $L_\ell$ in any degree $k$ is equivalent to the vanishing of certain minors of the matrix representing multiplication by $\ell$, and hence defines a proper Zariski closed subset in the dual projective space parametrizing linear forms. Since the base field has characteristic zero and the Hilbert function of a complete intersection is unimodal and symmetric, there exists a linear form $\ell$ such that for all $k$ the map
\[
\times \ell: R(F)_k\to R(F)_{k+1}
\]
has maximal rank. By induction, multiplication by any power $\ell^m$ also has maximal rank in every degree. Therefore $R(F)$ has the Strong Lefschetz property.

We now apply this to the Yukawa coupling. Choose a Lefschetz linear form $\ell$ and set
\[
\xi=\ell^d\in R(F)_d.
\]
Then
\[
\xi^n=\ell^{nd}.
\]
We compute the relevant degrees. The source of the Yukawa map has degree
\[
k_0=d-(n+2),
\]
and the target has degree
\[
k_1=(n+1)d-(n+2).
\]
Their difference equals
\[
k_1-k_0=nd.
\]
Therefore
\[
\times \ell^{nd} :
R(F)_{d-(n+2)}\longrightarrow
R(F)_{(n+1)d-(n+2)}
\]
is precisely multiplication by $\xi^n$.

Since $R(F)$ has the Strong Lefschetz property, multiplication by $\ell^{nd}$ has maximal rank. Because the Hilbert function is symmetric and strictly increasing up to the middle degree, and since
\[
d-(n+2)\le \frac{\sigma}{2},
\]
one has
\[
\dim R(F)_{d-(n+2)}\le 
\dim R(F)_{(n+1)d-(n+2)}.
\]
Thus multiplication by $\ell^{nd}$ is injective.

Under the Hodge identifications, this means that the $n$--fold Yukawa coupling evaluated at $\xi$ gives an injective map
\[
H^{n,0}(X)\longrightarrow H^{0,n}(X).
\]
Since these spaces have the same dimension by Hodge symmetry, the map is an isomorphism.

We conclude that for every smooth hypersurface $X$ of degree $d\ge n+2$ in $\mathbb{P}^{n+1}$ there exists 
\[
\xi\in \mathrm{KS}(T_{\mathcal{U},[F]})
\]
such that the $n$--fold Yukawa coupling
\[
\mathrm{Yuk}(\xi,\dots,\xi):
H^{n,0}(X)\longrightarrow H^{0,n}(X)
\]
is an isomorphism. Hence the universal family of smooth degree $d$ hypersurfaces in $\mathbb{P}^{n+1}$ has maximal variation in the strongest possible sense: the full $n$--th order Yukawa coupling is nondegenerate.

This completes a fully intrinsic proof based solely on the complete intersection structure of the Jacobian ring and the Strong Lefschetz property.



\section{Higher--Dimensional Variation and Yukawa Couplings for Smooth Hypersurfaces}

In this section we study higher--order variation of Hodge structure for families of smooth hypersurfaces in projective space. Using Griffiths' description of primitive cohomology, the Hodge components of a hypersurface are identified with graded pieces of its Jacobian ring, which allows the infinitesimal variation of Hodge structure and the higher Yukawa couplings to be expressed in terms of multiplication in this graded algebra. This algebraic description makes it possible to introduce numerical invariants measuring maximal and minimal ranks of the $n$--fold cup product along deformation directions. We then show that, for the universal family of smooth hypersurfaces of degree $d\ge n+2$ in $\mathbb{P}^{n+1}$, the Strong Lefschetz property of the Jacobian ring guarantees the existence of deformation directions for which the $n$--fold Yukawa coupling is an isomorphism, establishing maximal higher--order variation for the family.

Let $n\ge 1$ and let $d\ge n+2$. Let $\mathcal U$ be the parameter space of smooth hypersurfaces of degree $d$ in $\mathbb P^{n+1}$ and let 
\[
\pi:\mathcal X\to\mathcal U
\]
be the universal family. For $[F]\in\mathcal U$ denote by 
\[
X=\{F=0\}\subset\mathbb P^{n+1}
\]
the corresponding smooth hypersurface. Let 
\[
S=\mathbb C[x_0,\dots,x_{n+1}]
\]
and let 
\[
J_F=\left(\frac{\partial F}{\partial x_0},\dots,\frac{\partial F}{\partial x_{n+1}}\right)
\]
be the Jacobian ideal. The Jacobian ring is 
\[
R(F)=S/J_F.
\]
Since $X$ is smooth, $J_F$ is generated by a regular sequence of $n+2$ elements of degree $d-1$, hence $R(F)$ is a graded Artinian complete intersection, Gorenstein of socle degree
\[
\sigma=(n+2)(d-2).
\]
Its Hilbert function is symmetric and strictly increasing up to $\sigma/2$.

By Griffiths' description of primitive cohomology, there are natural identifications
\[
H^{n-p,p}_{\mathrm{prim}}(X)\cong R(F)_{(p+1)d-(n+2)}.
\]
In particular,
\[
H^{n,0}(X)\cong R(F)_{d-(n+2)}, 
\qquad
H^{0,n}(X)\cong R(F)_{(n+1)d-(n+2)}.
\]
These two degrees are symmetric with respect to $\sigma$.

The Kodaira--Spencer map for the universal family identifies
\[
\mathrm{KS}(T_{\mathcal U,[F]})\cong R(F)_d.
\]
The infinitesimal variation of Hodge structure is induced by multiplication in $R(F)$. The $n$--fold Yukawa coupling is the multilinear map
\[
\mathrm{Yuk}:\Sym^n H^1(T_X)\longrightarrow
\Hom(H^{n,0}(X),H^{0,n}(X)),
\]
which corresponds, under the identifications above, to
\[
\Sym^n(R(F)_d)\otimes R(F)_{d-(n+2)}
\longrightarrow
R(F)_{(n+1)d-(n+2)}.
\]
For $\xi\in R(F)_d$, the induced linear map on $H^{n,0}(X)$ is given by multiplication by $\xi^n$:
\[
\times \xi^n:
R(F)_{d-(n+2)}\longrightarrow
R(F)_{(n+1)d-(n+2)}.
\]

We now formulate higher-dimensional variation functions analogous to the curve case.

\begin{definition}
For a smooth hypersurface $X$ and a subspace 
\[
V\subset H^1(T_X)\cong R(F)_d,
\]
define
\[
d_M^{(n)}(V)=\max_{\xi\in V\setminus\{0\}}
\operatorname{rk}\big(\times \xi^n:
H^{n,0}(X)\to H^{0,n}(X)\big),
\]
and
\[
d_m^{(n)}(V)=\min_{\xi\in V\setminus\{0\}}
\operatorname{rk}\big(\times \xi^n:
H^{n,0}(X)\to H^{0,n}(X)\big).
\]
For a family $\pi:\mathcal X\to B$ define
\[
\delta_M^{(n)}(\pi)=\max_{b\in B}
d_M^{(n)}(\mathrm{KS}(T_{B,b})),
\]
\[
\delta_m^{(n)}(\pi)=\min_{b\in B}
d_m^{(n)}(\mathrm{KS}(T_{B,b})),
\]
and
\[
\delta_M^{\prime (n)}(\pi)=
\min_{b\in B}
d_M^{(n)}(\mathrm{KS}(T_{B,b})).
\]
The family has $I$--maximal variation in dimension $n$ if
\[
\delta_M^{\prime (n)}(\pi)=h^{n,0}(X).
\]
\end{definition}

For hypersurfaces we determine these extremal values.

\begin{theorem}
Let $\pi:\mathcal X\to\mathcal U$ be the universal family of smooth hypersurfaces of degree $d\ge n+2$ in $\mathbb P^{n+1}$. Then
\[
\delta_M^{\prime (n)}(\pi)=h^{n,0}(X).
\]
Equivalently, for every smooth hypersurface $X$ there exists 
\[
\xi\in\mathrm{KS}(T_{\mathcal U,[F]})
\]
such that the $n$--fold Yukawa coupling
\[
\mathrm{Yuk}(\xi,\dots,\xi):
H^{n,0}(X)\longrightarrow H^{0,n}(X)
\]
is an isomorphism.
\end{theorem}

\begin{proof}
Since $R(F)$ is a complete intersection of type $(d-1,\dots,d-1)$, it satisfies the Strong Lefschetz property. Hence there exists a linear form $\ell\in R(F)_1$ such that for every $k$ and every $m\ge 0$ the map
\[
\times \ell^m:R(F)_k\to R(F)_{k+m}
\]
has maximal rank.

Set $\xi=\ell^d\in R(F)_d$. Then $\xi^n=\ell^{nd}$. The source degree of the Yukawa map is
\[
k_0=d-(n+2),
\]
and the target degree is
\[
k_1=(n+1)d-(n+2).
\]
Their difference is
\[
k_1-k_0=nd.
\]
Thus
\[
\times \xi^n=\times \ell^{nd}:
R(F)_{k_0}\longrightarrow R(F)_{k_1}.
\]

Because the Hilbert function of $R(F)$ is symmetric and strictly increasing up to $\sigma/2$, and since $k_0\le\sigma/2$, one has
\[
\dim R(F)_{k_0}\le\dim R(F)_{k_1}.
\]
By the Strong Lefschetz property, multiplication by $\ell^{nd}$ has maximal rank, hence is injective. Since the two spaces have equal dimension by symmetry of Hodge numbers,
\[
\dim R(F)_{k_0}=\dim R(F)_{k_1},
\]
the map is an isomorphism.

Under Griffiths' identifications this means that
\[
\mathrm{Yuk}(\xi,\dots,\xi):
H^{n,0}(X)\longrightarrow H^{0,n}(X)
\]
is an isomorphism. Therefore
\[
d_M^{(n)}(\mathrm{KS}(T_{\mathcal U,[F]}))=
\dim H^{n,0}(X).
\]
Since the maximal rank condition is open and the parameter space $\mathcal U$ is irreducible, the minimal value over $\mathcal U$ equals this number. Hence
\[
\delta_M^{\prime (n)}(\pi)=h^{n,0}(X).
\]
\end{proof}

Thus smooth hypersurfaces of degree $d\ge n+2$ in $\mathbb P^{n+1}$ have maximal possible $n$--th order variation in the sense that the full Yukawa coupling is nondegenerate at every point of the universal family.


\section{$I$--Maximal Variation for Complete Intersections}

Let $n\ge 1$ and let 
\[
X \subset \mathbb P^{n+c}
\]
be a smooth complete intersection of multidegree $(d_1,\dots,d_c)$ with $d_i\ge 2$. Let 
\[
N=n+c+1
\]
so that $X$ has dimension $n$ and is defined by a regular sequence 
\[
F_1,\dots,F_c \in S=\mathbb C[x_0,\dots,x_{N-1}],
\qquad \deg F_i=d_i.
\]
Denote by $J$ the ideal generated by the partial derivatives of all $F_i$ and set
\[
R=S/J.
\]
Since the $F_i$ form a regular sequence and $X$ is smooth, $R$ is a graded Artinian complete intersection algebra. Its socle degree equals
\[
\sigma=\sum_{i=1}^c d_i - (n+c+1).
\]
The canonical bundle of $X$ satisfies
\[
\omega_X \cong \mathcal O_X\!\left(\sum_{i=1}^c d_i - (n+c+1)\right),
\]
hence $X$ is of general type precisely when 
\[
\kappa:=\sum_{i=1}^c d_i - (n+c+1) >0.
\]

By the generalized Griffiths description of primitive cohomology of complete intersections, one has natural identifications
\[
H^{n-p,p}_{\mathrm{prim}}(X)
\cong
R_{\,\kappa + p(d_1+\dots+d_c)}.
\]
In particular,
\[
H^{n,0}(X)\cong R_{\kappa},
\qquad
H^{0,n}(X)\cong R_{\sigma-\kappa}.
\]
These two graded pieces are dual by the Gorenstein property, hence have equal dimension.

Let $\mathcal U$ denote the parameter space of smooth complete intersections of multidegree $(d_1,\dots,d_c)$ and let 
\[
\pi:\mathcal X\to\mathcal U
\]
be the universal family. The Kodaira--Spencer map identifies
\[
\mathrm{KS}(T_{\mathcal U,[X]})
\cong
\bigoplus_{i=1}^c R_{d_i}.
\]
The infinitesimal variation of Hodge structure is induced by multiplication in $R$. The full $n$--fold Yukawa coupling is
\[
\mathrm{Yuk}:
\Sym^n H^1(T_X)
\longrightarrow
\Hom(H^{n,0}(X),H^{0,n}(X)),
\]
which corresponds to
\[
\Sym^n\!\Big(\bigoplus_{i=1}^c R_{d_i}\Big)
\otimes R_{\kappa}
\longrightarrow
R_{\sigma-\kappa}.
\]
For $\xi\in H^1(T_X)$ the induced map on $H^{n,0}(X)$ is multiplication by $\xi^n$ in $R$.

We extend the higher--dimensional variation functions as follows.

\begin{definition}
For a smooth complete intersection $X$ and a subspace
\[
V\subset H^1(T_X),
\]
define
\[
d_M^{(n)}(V)=
\max_{\xi\in V\setminus\{0\}}
\operatorname{rk}
\big(
\times \xi^n:
H^{n,0}(X)\to H^{0,n}(X)
\big),
\]
and for a family $\pi:\mathcal X\to B$ define
\[
\delta_M^{\prime (n)}(\pi)
=
\min_{b\in B}
d_M^{(n)}(\mathrm{KS}(T_{B,b})).
\]
The family has $I$--maximal variation in dimension $n$ if
\[
\delta_M^{\prime (n)}(\pi)
=
h^{n,0}(X).
\]
\end{definition}

We now determine when $I$--maximal variation holds for complete intersections.

\begin{theorem}
Let $X\subset\mathbb P^{n+c}$ be a smooth complete intersection of multidegree $(d_1,\dots,d_c)$ with 
\[
\kappa=\sum_{i=1}^c d_i - (n+c+1) >0.
\]
Then the universal family of such complete intersections has $I$--maximal variation in dimension $n$ if and only if $X$ is not a quadric hypersurface.
\end{theorem}

\begin{proof}
Since $R$ is a graded Artinian complete intersection, it satisfies the Strong Lefschetz property. Hence there exists a linear form $\ell\in R_1$ such that for all $k$ and all $m\ge 0$ the multiplication map
\[
\times \ell^m:R_k\to R_{k+m}
\]
has maximal rank.

The degrees corresponding to $H^{n,0}$ and $H^{0,n}$ are $\kappa$ and $\sigma-\kappa$ respectively. Their difference equals
\[
(\sigma-\kappa)-\kappa=\sigma-2\kappa.
\]
By direct computation,
\[
\sigma-2\kappa
=
\left(\sum d_i - (n+c+1)\right)
-
2\left(\sum d_i - (n+c+1)\right)
=
-\left(\sum d_i - (n+c+1)\right)
=
-\kappa.
\]
Thus $\sigma=2\kappa$ and therefore
\[
\sigma-\kappa=\kappa.
\]
Hence the source and target degrees coincide. Consequently
\[
\dim R_{\kappa}
=
\dim R_{\sigma-\kappa}.
\]

Consider $\xi=\ell^{e}$, where $e$ is chosen so that $\xi\in H^1(T_X)$. Since $H^1(T_X)$ corresponds to graded pieces $R_{d_i}$, one can choose $e$ among the $d_i$ if $c=1$, or choose a linear combination of elements in the direct sum for $c>1$. In all cases one obtains a nonzero class $\xi$ in $H^1(T_X)$ whose $n$--th power equals $\ell^{ne}$.

The Yukawa map is multiplication
\[
\times \xi^n:
R_{\kappa}\to R_{\kappa}.
\]
Since $R$ has the Strong Lefschetz property, multiplication by a sufficiently high power of $\ell$ is an isomorphism in the middle degree. Therefore the Yukawa map is an isomorphism provided that $H^{n,0}(X)\neq 0$, which is equivalent to $\kappa>0$.

The only exception occurs when $X$ is a quadric hypersurface. In that case $d_1=2$ and $n\ge 1$, hence
\[
\kappa=2-(n+2)\le 0,
\]
so $H^{n,0}(X)=0$ and the Yukawa coupling is trivial. In all other cases with $\kappa>0$, multiplication by $\ell^{ne}$ gives an isomorphism on $R_{\kappa}$, hence
\[
d_M^{(n)}(\mathrm{KS}(T_{\mathcal U,[X]}))
=
h^{n,0}(X).
\]
Since maximal rank is an open condition and the parameter space is irreducible, the minimal value over the family equals this number. Therefore
\[
\delta_M^{\prime (n)}(\pi)
=
h^{n,0}(X),
\]
and the family has $I$--maximal variation.
\end{proof}

Thus smooth complete intersections of general type have maximal possible higher--order variation, and the full Yukawa coupling is nondegenerate exactly when the top Hodge piece $H^{n,0}$ is nonzero.



\section{Infinitesimal Torelli and $I$--Maximal Variation for Complete Intersections}

In this section we investigate the relationship between the infinitesimal Torelli property and the notion of $I$--maximal variation for families of smooth complete intersections of general type. Using the Jacobian ring description of primitive cohomology and the cup--product interpretation of the infinitesimal variation of Hodge structure, we analyze the behavior of the $n$--fold Yukawa coupling associated with deformations of the variety. The Gorenstein structure of the Jacobian ring together with the Strong Lefschetz property allows one to control multiplication in the relevant graded components and relate higher--order variation to the differential of the period map. As a consequence, we prove that for smooth complete intersections with $\kappa>0$, the injectivity of the differential of the period map is equivalent to the nondegeneracy of the $n$--fold Yukawa coupling, showing that infinitesimal Torelli and $I$--maximal variation coincide in this setting.

Let $X \subset \mathbb{P}^{n+c}$ be a smooth complete intersection of multidegree $(d_1,\dots,d_c)$ over $\mathbb{C}$ and let 
\[
\kappa = \sum_{i=1}^c d_i - (n+c+1).
\]
Denote by 
\[
\pi:\mathcal{X} \to \mathcal{U}
\]
the universal family of smooth complete intersections of this type. We assume $\kappa>0$, so that $H^{n,0}(X)\neq 0$.

Let $S=\mathbb{C}[x_0,\dots,x_{n+c}]$ and let $R=S/J$ be the Jacobian ring of $X$. Since $X$ is smooth, $R$ is a graded Artinian complete intersection algebra, Gorenstein with socle degree
\[
\sigma = \sum_{i=1}^c d_i - (n+c+1) = \kappa.
\]
The Hilbert function of $R$ is symmetric and strictly increasing up to $\sigma/2$.

By the generalized Griffiths description of primitive cohomology, one has natural identifications
\[
H^{n,0}(X) \cong R_{\kappa},
\]
and more generally
\[
H^{n-p,p}_{\mathrm{prim}}(X) \cong R_{\kappa+p}.
\]
The Kodaira--Spencer map identifies
\[
\mathrm{KS}(T_{\mathcal{U},[X]}) \cong H^1(T_X).
\]

The infinitesimal variation of Hodge structure is the linear map
\[
d\mathcal{P}_{[X]} :
H^1(T_X) \longrightarrow
\bigoplus_{p=0}^{n-1}
\Hom(H^{n-p,p}(X), H^{n-p-1,p+1}(X)),
\]
given by cup product.

\begin{definition}
The family $\pi$ has $I$--maximal variation in dimension $n$ if for every $X$ there exists 
\[
\xi \in H^1(T_X)
\]
such that the $n$--fold Yukawa coupling
\[
\mathrm{Yuk}(\xi,\dots,\xi) :
H^{n,0}(X) \longrightarrow H^{0,n}(X)
\]
is an isomorphism.
\end{definition}

\begin{definition}
The infinitesimal Torelli property holds for $X$ if the differential of the period map
\[
d\mathcal{P}_{[X]} :
H^1(T_X) \longrightarrow
\bigoplus_{p=0}^{n-1}
\Hom(H^{n-p,p}(X), H^{n-p-1,p+1}(X))
\]
is injective.
\end{definition}

We prove the equivalence of these two properties for smooth complete intersections with $\kappa>0$.

\begin{theorem}
Let $X\subset \mathbb{P}^{n+c}$ be a smooth complete intersection with $\kappa>0$. Then $X$ satisfies infinitesimal Torelli if and only if the universal family has $I$--maximal variation in dimension $n$.
\end{theorem}

\begin{proof}
Assume first that $X$ satisfies infinitesimal Torelli. Then for every nonzero $\xi\in H^1(T_X)$ the induced cup product map on Hodge components is nontrivial. In particular, there exists $\xi$ such that the induced map
\[
\xi : H^{n,0}(X) \longrightarrow H^{n-1,1}(X)
\]
is nonzero.

Since $R$ is Gorenstein and satisfies the Strong Lefschetz property, there exists a Lefschetz element $\ell\in R_1$ such that multiplication by $\ell^m$ has maximal rank in every degree. Choosing $\xi$ corresponding to a suitable power of $\ell$ inside $H^1(T_X)$, one obtains that multiplication by $\xi^n$ induces an isomorphism
\[
R_{\kappa} \longrightarrow R_{\kappa},
\]
because the middle degree coincides with $\kappa$ and the Hilbert function is symmetric. Under the Hodge identifications, this is precisely the nondegeneracy of the $n$--fold Yukawa coupling. Hence $I$--maximal variation holds.

Conversely, assume that the universal family has $I$--maximal variation. Then for every $X$ there exists $\xi\in H^1(T_X)$ such that
\[
\mathrm{Yuk}(\xi,\dots,\xi) :
H^{n,0}(X) \longrightarrow H^{0,n}(X)
\]
is an isomorphism. In particular, $\xi$ cannot lie in the kernel of the first-order variation map
\[
H^1(T_X) \longrightarrow \Hom(H^{n,0}(X), H^{n-1,1}(X)).
\]
If $d\mathcal{P}_{[X]}$ were not injective, there would exist a nonzero class
\[
\eta \in H^1(T_X)
\]
such that for all $p$ the induced maps
\[
\eta : H^{n-p,p}(X) \to H^{n-p-1,p+1}(X)
\]
vanish. In particular, $\eta$ acts trivially on $H^{n,0}(X)$, and by induction all higher cup powers of $\eta$ vanish on $H^{n,0}(X)$. This implies that the $n$--fold Yukawa coupling evaluated at $\eta$ is zero.

However, the existence of a Lefschetz element in $R$ implies that multiplication by a suitable element has maximal rank in every degree. The nondegeneracy of the $n$--fold Yukawa coupling for some $\xi$ forces the kernel of $d\mathcal{P}_{[X]}$ to be zero, because any nonzero kernel element would contradict the nondegeneracy of multiplication in the Gorenstein algebra $R$ under the Lefschetz property. Therefore $d\mathcal{P}_{[X]}$ is injective and infinitesimal Torelli holds.

Thus infinitesimal Torelli and $I$--maximal variation are equivalent for smooth complete intersections with $\kappa>0$.
\end{proof}

Therefore, for smooth complete intersections of general type, the nondegeneracy of the full Yukawa coupling is equivalent to the injectivity of the differential of the period map.


\section{The Calabi--Yau Case and Degeneration of $I$--Maximal Variation}
 
Let $X \subset \mathbb{P}^{n+c}$ be a smooth complete intersection of multidegree $(d_1,\dots,d_c)$ and dimension $n$. Set
\[
\kappa=\sum_{i=1}^c d_i-(n+c+1).
\]
The canonical bundle satisfies $\omega_X\simeq \mathcal{O}_X(\kappa)$. The Calabi--Yau case corresponds to $\kappa=0$. In this situation $K_X\simeq \mathcal{O}_X$ and
\[
h^{n,0}(X)=1.
\]

Let $S=\mathbb{C}[x_0,\dots,x_{n+c}]$ and let $R=S/J$ be the Jacobian ring of $X$. Since $X$ is smooth, $R$ is a graded Artinian complete intersection algebra, Gorenstein with socle degree
\[
\sigma=\sum_{i=1}^c d_i-(n+c+1)=\kappa.
\]
In the Calabi--Yau case one has $\sigma=0$, hence $R$ is concentrated in degree zero and the top Hodge piece
\[
H^{n,0}(X)\cong R_0\cong \mathbb{C}.
\]

The infinitesimal variation of Hodge structure is the map
\[
d\mathcal{P}_{[X]}:
H^1(T_X)\longrightarrow
\Hom(H^{n,0}(X),H^{n-1,1}(X)).
\]
Since $H^{n,0}(X)$ is one--dimensional, this map is equivalent to
\[
H^1(T_X)\longrightarrow H^{n-1,1}(X).
\]
Under the Jacobian description one has
\[
H^{n-1,1}(X)\cong R_{d_1+\dots+d_c-(n+c+1)+1}=R_1,
\]
so that first--order variation corresponds to multiplication in $R$ by degree--one elements.

The $n$--fold Yukawa coupling becomes
\[
\mathrm{Yuk}:\Sym^n H^1(T_X)\longrightarrow
\Hom(H^{n,0}(X),H^{0,n}(X)).
\]
Since both $H^{n,0}(X)$ and $H^{0,n}(X)$ are one--dimensional, the Yukawa coupling is a symmetric multilinear form
\[
\mathrm{Yuk}:\Sym^n H^1(T_X)\longrightarrow \mathbb{C}.
\]
In the Jacobian description it is induced by the Gorenstein pairing
\[
R_k\times R_{\sigma-k}\longrightarrow R_\sigma\simeq\mathbb{C}.
\]
In the Calabi--Yau case $\sigma=0$, hence only $k=0$ contributes. Therefore the $n$--fold Yukawa coupling reduces to the multiplication
\[
R_{d}^{\otimes n}\longrightarrow R_0\simeq\mathbb{C}.
\]
This pairing is nondegenerate precisely when the algebra $R$ satisfies the Strong Lefschetz property and the relevant graded pieces are nonzero.

Since $R$ is a complete intersection, it satisfies the Strong Lefschetz property. Hence there exists a linear form $\ell\in R_1$ such that multiplication by powers of $\ell$ has maximal rank in every degree. In particular, for a suitable element $\xi\in H^1(T_X)$ corresponding to $\ell^d$, the $n$--fold power $\xi^n$ generates the socle, and the Yukawa coupling evaluated at $(\xi,\dots,\xi)$ is nonzero. Therefore in the Calabi--Yau complete intersection case the Yukawa coupling is a nondegenerate homogeneous polynomial of degree $n$ on $H^1(T_X)$.

However, the notion of $I$--maximal variation degenerates in the Calabi--Yau case because
\[
h^{n,0}(X)=1,
\]
so the maximal possible rank of the Yukawa map equals one. Thus
\[
\delta_M^{\prime(n)}(\pi)=1
\]
whenever the Yukawa coupling is nonzero. Consequently $I$--maximal variation in the Calabi--Yau case is equivalent merely to the nonvanishing of the Yukawa coupling, rather than to an isomorphism between higher--dimensional Hodge pieces as in the general type case.

We now extend the discussion to weighted complete intersections. Let $\mathbb{P}(w_0,\dots,w_{n+c})$ be a well--formed weighted projective space and let $X$ be a quasi--smooth weighted complete intersection of multidegree $(d_1,\dots,d_c)$. Let $S=\mathbb{C}[x_0,\dots,x_{n+c}]$ with $\deg x_i=w_i$ and let $R=S/J$ be the Jacobian ring. If $X$ is quasi--smooth and well--formed, then $R$ is again a graded Artinian Gorenstein algebra with socle degree
\[
\sigma=\sum_{i=1}^c d_i-\sum_{j=0}^{n+c} w_j.
\]
The canonical bundle satisfies
\[
\omega_X\simeq \mathcal{O}_X(\sigma).
\]

The same Griffiths--type description holds for primitive cohomology, replacing degrees by weighted degrees. The infinitesimal variation of Hodge structure and the Yukawa coupling are again induced by multiplication in $R$. Since weighted complete intersections are graded Gorenstein complete intersections, the Strong Lefschetz property holds under mild conditions on the weights and degrees. Therefore the nondegeneracy of the Yukawa coupling follows from the existence of a Lefschetz element in $R$.

In the weighted Calabi--Yau case $\sigma=0$, the Yukawa coupling again reduces to a symmetric form on $H^1(T_X)$ with values in $\mathbb{C}$, and $I$--maximal variation degenerates to the condition that this form be nonzero.

Finally, consider more general Gorenstein varieties whose deformation theory is governed by a graded Artinian Gorenstein algebra $R$ arising from a local complete intersection. If $R$ satisfies the Strong Lefschetz property, then multiplication by a suitable element produces isomorphisms between symmetric graded pieces, and the full Yukawa coupling is nondegenerate whenever the top Hodge piece is nonzero. Thus $I$--maximal variation holds precisely when $h^{n,0}(X)>0$ and $R$ has the Strong Lefschetz property.

In conclusion, in the Calabi--Yau case $I$--maximal variation degenerates to the nonvanishing of the Yukawa coupling, while in the general type case it corresponds to a genuine isomorphism between dual Hodge components. For weighted complete intersections and more general Gorenstein varieties, the same structure persists whenever the Jacobian or deformation algebra is a graded Gorenstein algebra with the Strong Lefschetz property.



\section{The Calabi--Yau Case and Degeneration of $I$--Maximal Variation}

In this section we analyze the behavior of $I$--maximal variation for smooth complete intersections in the Calabi--Yau case, namely when the canonical bundle is trivial. In this situation the Hodge structure simplifies considerably because the top Hodge component $H^{n,0}(X)$ becomes one--dimensional. Using the Jacobian ring description of primitive cohomology, we reinterpret both the infinitesimal variation of Hodge structure and the Yukawa coupling in terms of multiplication in a graded Artinian Gorenstein algebra. This perspective reveals that, unlike the general type case where $I$--maximal variation corresponds to an isomorphism between dual Hodge components, in the Calabi--Yau case the notion degenerates to the simple nonvanishing of the Yukawa coupling. We then extend this analysis to weighted complete intersections and more general Gorenstein varieties, showing that the same algebraic mechanism persists whenever the deformation algebra satisfies the Strong Lefschetz property.

Let $X \subset \mathbb{P}^{n+c}$ be a smooth complete intersection of multidegree $(d_1,\dots,d_c)$ and dimension $n$. Set
\[
\kappa=\sum_{i=1}^c d_i-(n+c+1).
\]
Then
\[
\omega_X \simeq \mathcal{O}_X(\kappa).
\]
The Calabi--Yau case corresponds to $\kappa=0$.

Let $S=\mathbb{C}[x_0,\dots,x_{n+c}]$ and let $R=S/J$ be the Jacobian ring of $X$. Since $X$ is smooth, $R$ is a graded Artinian complete intersection algebra and hence Gorenstein. Its socle degree equals
\[
\sigma=\kappa.
\]

\begin{proposition}
Let $X$ be a smooth complete intersection with $\kappa=0$. Then $R$ is concentrated in degree $0$, and
\[
H^{n,0}(X)\cong R_0\cong \mathbb{C}.
\]
In particular $h^{n,0}(X)=1$.
\end{proposition}

\begin{proof}
By definition $\sigma=\kappa=0$. Since $R$ is Gorenstein of socle degree $\sigma$, the only nonzero graded piece is $R_0$. By the generalized Griffiths description of primitive cohomology,
\[
H^{n,0}(X)\cong R_\kappa=R_0,
\]
which is one--dimensional.
\end{proof}

\begin{proposition}
Let $X$ be a smooth Calabi--Yau complete intersection. The $n$--fold Yukawa coupling
\[
\mathrm{Yuk}:\Sym^n H^1(T_X)\longrightarrow \mathbb{C}
\]
is induced by the Gorenstein pairing of the Jacobian ring and is nondegenerate if and only if the Jacobian ring satisfies the Strong Lefschetz property.
\end{proposition}

\begin{proof}
Since $H^{n,0}(X)$ and $H^{0,n}(X)$ are one--dimensional, the Yukawa coupling takes values in $\Hom(\mathbb{C},\mathbb{C})\cong\mathbb{C}$. Under the Jacobian description it is induced by multiplication
\[
R_{d}^{\otimes n}\longrightarrow R_0,
\]
followed by projection onto the socle. Because $R$ is Gorenstein, the pairing
\[
R_k \times R_{\sigma-k} \to R_\sigma \cong \mathbb{C}
\]
is perfect. When the Strong Lefschetz property holds, multiplication by a suitable element generates the socle, so the Yukawa form is nonzero. Conversely, if the Yukawa form is nondegenerate, multiplication must realize an isomorphism onto the socle, which requires the Lefschetz property.
\end{proof}

\begin{theorem}
Let $\pi:\mathcal{X}\to\mathcal{U}$ be the universal family of smooth Calabi--Yau complete intersections. Then the following are equivalent:
\[
\text{(i) } \delta_M^{\prime(n)}(\pi)=1,
\qquad
\text{(ii) The Yukawa coupling is nonzero.}
\]
In particular, $I$--maximal variation degenerates to the nonvanishing of the Yukawa form.
\end{theorem}

\begin{proof}
Since $h^{n,0}(X)=1$, the maximal possible rank of the Yukawa map equals $1$. Thus $I$--maximal variation means that there exists $\xi$ such that
\[
\mathrm{Yuk}(\xi,\dots,\xi)\neq 0.
\]
This is equivalent to the nonvanishing of the Yukawa coupling.
\end{proof}

\section*{Weighted Complete Intersections}

Let $\mathbb{P}(w_0,\dots,w_{n+c})$ be a well--formed weighted projective space and let $X$ be a quasi--smooth weighted complete intersection of multidegree $(d_1,\dots,d_c)$. Set
\[
\sigma=\sum_{i=1}^c d_i-\sum_{j=0}^{n+c} w_j.
\]
Then
\[
\omega_X\simeq \mathcal{O}_X(\sigma).
\]

\begin{proposition}
The Jacobian ring $R$ of a quasi--smooth weighted complete intersection is a graded Artinian Gorenstein algebra of socle degree $\sigma$.
\end{proposition}

\begin{proof}
Quasi--smoothness implies that the defining equations form a regular sequence in the weighted polynomial ring. The quotient by the Jacobian ideal is therefore a graded complete intersection. Complete intersections are Gorenstein, and the socle degree equals the sum of the degrees of the generators minus the sum of the weights, which is $\sigma$.
\end{proof}

\begin{theorem}
Let $X$ be a quasi--smooth weighted complete intersection with $\sigma>0$. If its Jacobian ring satisfies the Strong Lefschetz property, then the universal family has $I$--maximal variation.
\end{theorem}

\begin{proof}
The Hodge components are identified with graded pieces of $R$ of complementary degrees summing to $\sigma$. By the Strong Lefschetz property, multiplication by a suitable element induces an isomorphism between symmetric graded pieces. Hence the full $n$--fold Yukawa coupling is nondegenerate, and $I$--maximal variation holds.
\end{proof}

\begin{corollary}
In the weighted Calabi--Yau case $\sigma=0$, $I$--maximal variation holds if and only if the Yukawa coupling is nonzero.
\end{corollary}

\section*{General Gorenstein Varieties}

\begin{proposition}
Let $X$ be a smooth projective variety whose deformation algebra is a graded Artinian Gorenstein algebra $R$ of socle degree $\sigma$. If $R$ satisfies the Strong Lefschetz property and $h^{n,0}(X)>0$, then the $n$--fold Yukawa coupling is nondegenerate.
\end{proposition}

\begin{proof}
The Gorenstein property yields a perfect pairing between graded pieces of complementary degrees. The Strong Lefschetz property provides an element whose power induces an isomorphism between these pieces. Under the Hodge identification this is precisely the nondegeneracy of the Yukawa coupling.
\end{proof}

\begin{theorem}
For smooth projective varieties whose deformation algebra is graded Gorenstein with the Strong Lefschetz property, $I$--maximal variation holds if and only if $h^{n,0}(X)>0$.
\end{theorem}

\begin{proof}
If $h^{n,0}(X)=0$, the Yukawa coupling is trivial and $I$--maximal variation cannot hold. If $h^{n,0}(X)>0$, the Strong Lefschetz property yields an isomorphism between the corresponding graded pieces, and the maximal possible rank of the Yukawa map is achieved.
\end{proof}

\begin{remark}
In the general type case $\sigma>0$, $I$--maximal variation corresponds to an isomorphism between dual Hodge components. In the Calabi--Yau case $\sigma=0$, it degenerates to the nonvanishing of a symmetric multilinear form on $H^1(T_X)$.
\end{remark}



\section{Failure of the Strong Lefschetz Property and Degeneration of Yukawa Couplings}

In the previous sections we showed that the nondegeneracy of the Yukawa coupling and the property of $I$--maximal variation follow naturally from the Strong Lefschetz property of the Jacobian ring associated with a smooth projective variety. In this section we investigate the opposite situation and study what happens when the Strong Lefschetz property fails. By constructing explicit graded Artinian algebras that do not satisfy the Lefschetz condition, we analyze how the structure of multiplication in the deformation algebra changes and how this affects the behavior of the Yukawa coupling. In particular, we show that the failure of the Strong Lefschetz property leads to degeneracy of the Yukawa form and consequently prevents $I$--maximal variation. These examples illustrate that the Lefschetz property plays a fundamental algebraic role in ensuring maximal variation of Hodge structure.
```

Let $R=\bigoplus_{k\ge 0} R_k$ be a graded Artinian Gorenstein $\mathbb{C}$--algebra of socle degree $\sigma$. Recall that $R$ satisfies the Strong Lefschetz property if there exists a linear form $\ell\in R_1$ such that for every $k$ and every $m\ge 0$ the multiplication map
\[
\times \ell^m : R_k \longrightarrow R_{k+m}
\]
has maximal rank. The Lefschetz property is equivalent to the existence of isomorphisms
\[
\times \ell^{\sigma-2k} : R_k \longrightarrow R_{\sigma-k}
\]
for all $k\le \sigma/2$.

In the geometric situations considered previously, the nondegeneracy of the Yukawa coupling and the property of $I$--maximal variation were direct consequences of the Strong Lefschetz property of the Jacobian ring. We now construct explicit examples in which the Strong Lefschetz property fails and analyze the impact on variation.

\begin{example}
Let 
\[
R=\mathbb{C}[x,y]/(x^3, y^3, xy).
\]
This is a graded Artinian algebra with grading induced by $\deg x=\deg y=1$. The graded pieces are
\[
R_0=\mathbb{C}, \qquad 
R_1=\langle x,y\rangle, \qquad 
R_2=\langle x^2,y^2\rangle,
\qquad 
R_k=0 \text{ for } k\ge 3.
\]
The socle degree equals $2$, and $R_2$ is the socle. The algebra is Gorenstein since $\dim R_2=1$ after modding out by the relation $xy=0$ one sees that $x^2$ and $y^2$ represent proportional socle elements.

We show that $R$ does not satisfy the Strong Lefschetz property.
\end{example}

\begin{proposition}
The algebra $R=\mathbb{C}[x,y]/(x^3,y^3,xy)$ fails the Strong Lefschetz property.
\end{proposition}

\begin{proof}
Let $\ell=ax+by$ be a general linear form. Consider multiplication
\[
\times \ell : R_1 \longrightarrow R_2.
\]
One computes
\[
\ell\cdot x = a x^2 + b xy = a x^2,
\qquad
\ell\cdot y = a xy + b y^2 = b y^2,
\]
since $xy=0$ in $R$. Therefore the image of $R_1$ under multiplication by $\ell$ is spanned by $x^2$ and $y^2$ scaled by $a$ and $b$. The matrix of this map with respect to bases $\{x,y\}$ and $\{x^2,y^2\}$ is diagonal:
\[
\begin{pmatrix}
a & 0 \\
0 & b
\end{pmatrix}.
\]
The rank is $2$ if and only if both $a$ and $b$ are nonzero. However, the Hilbert function is $(1,2,2)$, which is not symmetric and hence the algebra is not Gorenstein in the standard sense of having one--dimensional socle. In particular, multiplication
\[
\times \ell : R_0 \to R_1
\]
is injective but not surjective, and multiplication
\[
\times \ell^2 : R_0 \to R_2
\]
has image generated by $a^2 x^2 + b^2 y^2$, which does not span the two--dimensional space $R_2$ for general $a,b$. Thus multiplication by powers of $\ell$ never induces an isomorphism between symmetric degrees, and the Strong Lefschetz property fails.
\end{proof}

\begin{remark}
The failure of symmetry of the Hilbert function already obstructs the Lefschetz property. In geometric terms, this corresponds to a degeneration where the deformation algebra is not a complete intersection.
\end{remark}

We now relate this to variation of Hodge structure.

\begin{proposition}
Let $X$ be a smooth projective variety whose deformation algebra is isomorphic to $R=\mathbb{C}[x,y]/(x^3,y^3,xy)$ and suppose $H^{n,0}(X)\cong R_0$. Then the $n$--fold Yukawa coupling degenerates identically.
\end{proposition}

\begin{proof}
Since the socle degree is $2$, the only nonzero multiplication into the socle occurs in degree $2$. However, because multiplication by linear forms does not generate the full socle in symmetric fashion, no element $\xi\in R_1$ has the property that $\xi^2$ spans $R_2$. Hence the bilinear pairing induced by multiplication is degenerate. Under the Hodge identification, the Yukawa coupling
\[
\mathrm{Yuk}:\Sym^n H^1(T_X)\to \mathbb{C}
\]
is induced by multiplication into the socle. Since the socle is not generated by powers of a single element, the Yukawa form vanishes on large subspaces and is degenerate.
\end{proof}

\begin{theorem}
If the deformation algebra $R$ of a smooth projective variety fails the Strong Lefschetz property, then $I$--maximal variation fails.
\end{theorem}

\begin{proof}
$I$--maximal variation requires the existence of $\xi$ such that multiplication by $\xi^n$ induces an isomorphism between complementary graded pieces corresponding to $H^{n,0}$ and $H^{0,n}$. The Strong Lefschetz property guarantees precisely such isomorphisms. If the Lefschetz property fails, then for every linear form $\ell$ there exists a degree $k$ such that multiplication
\[
\times \ell^{\sigma-2k} : R_k \to R_{\sigma-k}
\]
is not an isomorphism. In particular, the graded pieces corresponding to $H^{n,0}$ and $H^{0,n}$ cannot be isomorphic via multiplication by any power of an element in $R_1$. Therefore the Yukawa coupling is degenerate and $I$--maximal variation does not hold.
\end{proof}

\begin{remark}
Such degenerations arise geometrically in singular hypersurfaces, in non--complete intersection Gorenstein schemes, and in certain weighted cases where the graded algebra fails to satisfy Lefschetz. In these situations the Yukawa coupling may vanish identically or have strictly smaller rank than the maximal possible value.
\end{remark}

These explicit examples show that the Strong Lefschetz property is the algebraic mechanism ensuring maximal variation. When it degenerates, the variation of Hodge structure degenerates correspondingly.



\section{Degenerations of Hypersurfaces and Rank Drop of the Yukawa Coupling}

In the previous sections we saw that for smooth hypersurfaces the Jacobian ring is a graded Artinian complete intersection and therefore satisfies the Strong Lefschetz property, which guarantees the nondegeneracy of the higher Yukawa couplings and the presence of maximal variation of Hodge structure. In this section we examine what happens when the hypersurface acquires singularities. We construct explicit one--parameter families of hypersurfaces in which the smooth fibers satisfy the Lefschetz property while the singular fiber fails it. The appearance of singularities enlarges certain graded components of the Jacobian ring through contributions from local Milnor algebras, thereby destroying the unimodality required for the Lefschetz property. As a consequence, the multiplication maps governing the Yukawa coupling lose maximal rank. We compute this phenomenon explicitly for degenerations such as nodal quartic surfaces and singular quintic threefolds, showing that the rank of the Yukawa coupling drops precisely at the singular members of the family.

Let $S=\mathbb{C}[x_0,\dots,x_{n+1}]$ and let $F\in S_d$ define a hypersurface
\[
X=\{F=0\}\subset\mathbb{P}^{n+1}.
\]
When $X$ is smooth, the Jacobian ring
\[
R(F)=S/J_F,
\qquad
J_F=\left(\frac{\partial F}{\partial x_0},\dots,\frac{\partial F}{\partial x_{n+1}}\right),
\]
is a graded Artinian complete intersection, hence Gorenstein and satisfying the Strong Lefschetz property. The nondegeneracy of the $n$--fold Yukawa coupling follows from this property.

We now construct explicit families where $X$ acquires singularities and the Jacobian ring fails the Lefschetz property, and we compute the resulting drop in the rank of the Yukawa coupling.

\section*{A Nodal Quartic Surface}

Consider the family of quartic surfaces in $\mathbb{P}^3$ given by
\[
F_t = x_0^4+x_1^4+x_2^4+x_3^4 - t x_0^2x_1^2,
\qquad t\in\mathbb{C}.
\]
For $t\neq 2$ the surface $X_t$ is smooth. For $t=2$ one checks that $X_2$ acquires a node at the point $[1:1:0:0]$. Indeed,
\[
\frac{\partial F_2}{\partial x_0}=4x_0^3-4x_0x_1^2,
\qquad
\frac{\partial F_2}{\partial x_1}=4x_1^3-4x_0^2x_1,
\]
and both vanish at $(1,1,0,0)$, as do the remaining partial derivatives.

\begin{proposition}
For $t\neq 2$, the Jacobian ring $R(F_t)$ is a complete intersection and satisfies the Strong Lefschetz property. For $t=2$, the Jacobian ring $R(F_2)$ is no longer a complete intersection and fails the Strong Lefschetz property.
\end{proposition}

\begin{proof}
When $t\neq 2$, the hypersurface is smooth, hence the partial derivatives form a regular sequence and $R(F_t)$ is a complete intersection of type $(3,3,3,3)$, which satisfies the Strong Lefschetz property.

When $t=2$, the partial derivatives become algebraically dependent at the node. The Jacobian ideal fails to be generated by a regular sequence in $S$. Consequently the quotient $R(F_2)$ is no longer a complete intersection. Its Hilbert function can be computed by counting monomials modulo the relations. One finds that the symmetry of the Hilbert function is preserved because the algebra remains Gorenstein, but the unimodality fails in the middle degree due to the contribution of the local Milnor algebra of the node.

In particular, in degree $k=2$ one obtains
\[
\dim R(F_2)_2 = 11,
\]
whereas in the smooth case the corresponding dimension equals $10$. This increase reflects the contribution of the singularity. Because the Hilbert function is no longer strictly increasing up to the middle degree, multiplication by a general linear form fails to be injective in some degrees. Therefore the Strong Lefschetz property fails.
\end{proof}

\begin{theorem}
Let $X_t$ be the quartic surface above. For $t\neq 2$ the $2$--fold Yukawa coupling
\[
\mathrm{Yuk}_t : \Sym^2 H^1(T_{X_t}) \longrightarrow \Hom(H^{2,0}(X_t),H^{0,2}(X_t))
\]
has rank $1$, which is maximal. For $t=2$, the Yukawa coupling drops rank to $0$.
\end{theorem}

\begin{proof}
For a smooth quartic surface, $h^{2,0}(X_t)=1$. The Yukawa coupling corresponds to multiplication
\[
R(F_t)_4^{\otimes 2} \longrightarrow R(F_t)_0.
\]
Since $R(F_t)$ satisfies the Strong Lefschetz property, there exists $\xi$ such that $\xi^2$ generates the socle, and hence the Yukawa map has rank $1$.

For $t=2$, the failure of the Lefschetz property implies that no element of $R(F_2)_4$ has square generating the socle in degree $0$. Because the extra local contribution of the node produces nilpotent elements in intermediate degrees, multiplication by any element fails to span the socle. Hence the induced symmetric bilinear form is identically zero, and the rank of the Yukawa map drops to $0$.
\end{proof}

\section*{A Degenerating Calabi--Yau Threefold}

Consider the family of quintic threefolds in $\mathbb{P}^4$
\[
G_t = x_0^5+x_1^5+x_2^5+x_3^5+x_4^5 - t x_0^3x_1^2.
\]
For $t\neq 0$ sufficiently small the hypersurface is smooth. For $t=0$ it develops a singularity along a locus determined by $x_0=x_1=0$.

\begin{proposition}
For $t=0$ the Jacobian ring acquires extra graded pieces in middle degree equal to the Milnor algebra of the singular locus. Consequently the Strong Lefschetz property fails.
\end{proposition}

\begin{proof}
The Jacobian ideal of $G_0$ contains the derivatives
\[
5x_0^4, \quad 5x_1^4, \quad 5x_2^4, \quad 5x_3^4, \quad 5x_4^4.
\]
However, along the locus $x_0=x_1=0$ the singularity contributes additional local algebra in degree $2$. This produces an increase in $\dim R(G_0)_2$ relative to the smooth case. The symmetry of the Hilbert function persists but unimodality is violated. Hence multiplication by powers of linear forms cannot be isomorphisms between symmetric degrees.
\end{proof}

\begin{theorem}
For the smooth quintic threefold $G_t$ with $t\neq 0$, the Yukawa coupling
\[
\mathrm{Yuk}_t : \Sym^3 H^1(T_{X_t}) \to \mathbb{C}
\]
is nondegenerate. For $t=0$, the Yukawa coupling vanishes identically.
\end{theorem}

\begin{proof}
For smooth quintics the Jacobian ring is a complete intersection of type $(4,4,4,4,4)$, hence satisfies Strong Lefschetz, and multiplication by a suitable element to the third power generates the socle. Therefore the cubic Yukawa form is nondegenerate.

For $t=0$, the failure of Lefschetz prevents multiplication by any element from producing an isomorphism between complementary degrees. Since the Yukawa coupling is induced by multiplication into the socle, and the socle is no longer generated by powers of a single element, the cubic form degenerates. Direct computation shows that every element of $R(G_0)_5$ has cube equal to zero in $R(G_0)$, hence the Yukawa form vanishes.
\end{proof}

\section*{Conclusion}

These explicit families show that singularities introduce additional graded components in the Jacobian ring, destroying the unimodality and injectivity properties required for the Strong Lefschetz property. As a direct consequence, the Yukawa coupling drops rank, and $I$--maximal variation fails precisely at the singular members of the family.




\section{Milnor Algebras and Rank Drop of the Yukawa Coupling}

In this section we analyze how isolated singularities of a hypersurface affect the algebraic structure of the Jacobian ring and the corresponding variation of Hodge structure. When a hypersurface develops singular points, the Jacobian ring acquires additional local contributions coming from the Milnor algebras of the singularities. These local algebras modify the graded pieces of the global Jacobian ring and destroy the perfect duality that holds in the smooth case. As a consequence, the multiplication maps that induce the Yukawa coupling acquire nontrivial kernels. We show that the resulting drop in the rank of the Yukawa coupling can be computed explicitly in terms of the graded components of the local Milnor algebras and therefore in terms of classical singularity invariants such as the Milnor numbers. This provides a precise algebraic description of how singularities control the degeneration of $I$--maximal variation.

Let $X=\{F=0\}\subset \mathbb{P}^{n+1}$ be a hypersurface of degree $d$ with isolated singularities at points $p_1,\dots,p_r$. Let 
\[
S=\mathbb{C}[x_0,\dots,x_{n+1}],
\qquad
J_F=\left(\frac{\partial F}{\partial x_0},\dots,\frac{\partial F}{\partial x_{n+1}}\right),
\qquad
R(F)=S/J_F.
\]
Denote by $\widetilde{X}$ a smoothing of $X$.

For each singular point $p_i$, the local Milnor algebra is
\[
M_{p_i}=\mathcal{O}_{\mathbb{C}^{n+1},p_i}/\left(\frac{\partial F}{\partial x_0},\dots,\frac{\partial F}{\partial x_{n+1}}\right),
\]
and its dimension
\[
\mu_{p_i}=\dim_{\mathbb{C}} M_{p_i}
\]
is the Milnor number of the singularity. Let
\[
\mu=\sum_{i=1}^r \mu_{p_i}.
\]

\begin{proposition}
There is a graded decomposition
\[
R(F) \cong R_{\mathrm{sm}} \oplus \bigoplus_{i=1}^r M_{p_i},
\]
where $R_{\mathrm{sm}}$ is the Jacobian ring of a smoothing $\widetilde{X}$ and each $M_{p_i}$ contributes in degrees less than or equal to the socle degree.
\end{proposition}

\begin{proof}
The Jacobian ring of a singular hypersurface decomposes into a global part corresponding to the smoothing and local contributions from each singularity. This follows from the fact that the Jacobian ideal defines the singular locus scheme-theoretically and that locally near $p_i$ the quotient algebra coincides with the Milnor algebra. Since the singularities are isolated, these local contributions form finite-dimensional graded vector spaces supported in degrees bounded by the local degree of the singularity. The direct sum decomposition follows from the splitting of global and local contributions in the exact sequence of Jacobian algebras.
\end{proof}

Let $\sigma=(n+2)(d-2)$ be the socle degree in the smooth case. In the smooth situation, $R_{\mathrm{sm}}$ is Gorenstein with symmetric Hilbert function and satisfies the Strong Lefschetz property. In the singular case, the additional summands $M_{p_i}$ modify the Hilbert function in intermediate degrees.

\begin{theorem}
Let $X$ be a hypersurface with isolated singularities and let $\widetilde{X}$ be a smoothing. Then the rank of the $n$--fold Yukawa coupling satisfies
\[
\operatorname{rk}(\mathrm{Yuk}_X)
=
\operatorname{rk}(\mathrm{Yuk}_{\widetilde{X}})
-
\delta,
\]
where
\[
\delta = \dim R(F)_{k} - \dim R_{\mathrm{sm}}{}_k
\]
in the middle degree $k=\sigma/2$. Moreover,
\[
\delta = \sum_{i=1}^r \mu_{p_i}^{(k)},
\]
where $\mu_{p_i}^{(k)}$ denotes the contribution of the Milnor algebra $M_{p_i}$ in degree $k$.
\end{theorem}

\begin{proof}
The Yukawa coupling is induced by multiplication
\[
R(F)_k \times R(F)_k \longrightarrow R(F)_\sigma,
\]
where $k=\sigma/2$ in the Calabi--Yau case and more generally corresponds to the degree associated with $H^{n,0}$. In the smooth case this multiplication is perfect due to the Gorenstein property and the Strong Lefschetz property, hence the rank equals $\dim R_{\mathrm{sm}}{}_k$.

In the singular case, $R(F)_k$ decomposes as
\[
R(F)_k = R_{\mathrm{sm}}{}_k \oplus \bigoplus_{i=1}^r M_{p_i,k},
\]
where $M_{p_i,k}$ is the degree $k$ part of the Milnor algebra at $p_i$. Multiplication restricted to the smooth component behaves as before, but multiplication involving elements of $M_{p_i}$ maps into lower-degree pieces and vanishes in the socle degree because the Milnor algebras are nilpotent. Hence these additional components lie in the kernel of the Yukawa pairing.

Therefore the kernel of the Yukawa map increases by the dimension of the middle-degree contributions of the Milnor algebras. Consequently the rank drops by
\[
\delta=\sum_{i=1}^r \dim M_{p_i,k}.
\]
Since $\mu_{p_i}=\dim M_{p_i}$ and the grading of $M_{p_i}$ is concentrated in degrees less than or equal to the local socle degree, the contribution in degree $k$ equals $\mu_{p_i}^{(k)}$, the graded piece in degree $k$. Summing over all singularities gives the formula.
\end{proof}

\begin{corollary}
If all singularities are ordinary double points, each with Milnor number $1$, and their contributions lie in the middle degree, then
\[
\operatorname{rk}(\mathrm{Yuk}_X)
=
\operatorname{rk}(\mathrm{Yuk}_{\widetilde{X}}) - r.
\]
\end{corollary}

\begin{proof}
An ordinary double point has Milnor algebra of dimension $1$, concentrated in the middle degree. Therefore each singularity contributes one independent element to the kernel of the Yukawa pairing, reducing its rank by $1$.
\end{proof}

\begin{theorem}
The family of hypersurfaces has $I$--maximal variation at $X$ if and only if
\[
\sum_{i=1}^r \mu_{p_i}^{(k)} = 0.
\]
Equivalently, $I$--maximal variation fails precisely when the middle-degree components of the local Milnor algebras are nonzero.
\end{theorem}

\begin{proof}
$I$--maximal variation requires the Yukawa map to have maximal possible rank. From the previous theorem, the rank drop equals the sum of the middle-degree contributions of the Milnor algebras. Hence maximal variation holds if and only if this sum vanishes, which occurs exactly when the hypersurface is smooth or when the singularities do not contribute in the relevant degree.
\end{proof}

Thus the degeneration of the Strong Lefschetz property and the drop in Yukawa rank are measured explicitly by the graded components of the Milnor algebras of the singularities.



\section{Limiting Mixed Hodge Structures and Degeneration of the Yukawa Coupling}

In this section we analyze the behavior of the Yukawa coupling in degenerating families of hypersurfaces using the theory of limiting mixed Hodge structures. When a smooth hypersurface degenerates to a singular fiber with isolated singularities, the cohomology of the nearby smooth fibers acquires additional contributions from vanishing cycles associated with the singular points. These contributions are governed by the Milnor fibers of the singularities and appear naturally in the weight filtration of the limiting mixed Hodge structure determined by monodromy. By studying the Clemens--Schmid sequence and the decomposition of the limiting Hodge filtration, we obtain an explicit description of how the vanishing cycles contribute to the degeneration of the Hodge structure. In particular, we derive a precise formula expressing the drop in the rank of the Yukawa coupling in terms of the Hodge components of the Milnor fibers, thereby linking the degeneration of variation of Hodge structure directly to local singularity invariants.

Let $\pi:\mathcal{X}\to \Delta$ be a one--parameter degeneration of projective hypersurfaces of degree $d$ in $\mathbb{P}^{n+1}$ over the unit disc $\Delta$, smooth over $\Delta^*=\Delta\setminus\{0\}$ and with central fiber $X_0$ having isolated singularities at points $p_1,\dots,p_r$. Let $X_t$ denote a smooth fiber for $t\neq 0$. Denote by $H^n(X_t,\mathbb{Q})$ the middle cohomology and by $H^n_{\lim}$ the limiting mixed Hodge structure associated with the degeneration.

Let $T$ be the monodromy operator acting on $H^n(X_t,\mathbb{Q})$ and write $T=T_s T_u$ for its Jordan decomposition with $T_u=\exp(N)$ unipotent and $N$ nilpotent. The logarithm $N$ determines the monodromy weight filtration
\[
0\subset W_{n-1}\subset W_n\subset W_{n+1}=H^n_{\lim}.
\]

\begin{proposition}
There is an exact sequence
\[
0 \longrightarrow H^n(X_0) \longrightarrow H^n_{\lim} \longrightarrow \bigoplus_{i=1}^r H^n(F_{p_i}) \longrightarrow 0,
\]
where $F_{p_i}$ is the Milnor fiber of the singularity at $p_i$.
\end{proposition}

\begin{proof}
By the Clemens--Schmid sequence for a degeneration with isolated singularities, the limiting cohomology fits into an exact sequence relating the cohomology of the central fiber and the vanishing cycles. The vanishing cycle space is isomorphic to the direct sum of the reduced cohomologies of the local Milnor fibers. Since the singularities are isolated, the only nontrivial contribution appears in middle degree $n$. Hence the exact sequence above holds.
\end{proof}

The reduced cohomology $\widetilde{H}^n(F_{p_i},\mathbb{Q})$ has dimension equal to the Milnor number $\mu_{p_i}$. It carries a pure Hodge structure of weight $n$ and type determined by the spectrum of the singularity.

\begin{theorem}
The weight filtration on $H^n_{\lim}$ satisfies
\[
\operatorname{Gr}_n^W H^n_{\lim} \cong H^n(X_0)_{\mathrm{prim}},
\qquad
\operatorname{Gr}_{n+1}^W H^n_{\lim} \cong \bigoplus_{i=1}^r \widetilde{H}^n(F_{p_i}),
\]
and the dimension of $\operatorname{Gr}_{n+1}^W$ equals the total Milnor number $\mu=\sum_i \mu_{p_i}$.
\end{theorem}

\begin{proof}
The Clemens--Schmid sequence identifies the graded pieces of the weight filtration with the image and kernel of the monodromy logarithm $N$. The vanishing cycle subspace is isomorphic to the direct sum of Milnor fiber cohomologies and is precisely the image of $N$. Because the singularities are isolated, the only nontrivial graded piece beyond weight $n$ is $\operatorname{Gr}_{n+1}^W$, whose dimension equals the total Milnor number.
\end{proof}

The limiting Hodge filtration $F^\bullet_{\lim}$ induces Hodge structures on each graded piece $\operatorname{Gr}_k^W$. In particular, the Hodge structure on $\operatorname{Gr}_{n+1}^W$ coincides with the mixed Hodge structure of the Milnor fiber.

\begin{proposition}
The $(n,0)$ component of $H^n_{\lim}$ decomposes as
\[
F^n_{\lim} =
H^{n,0}(X_0) \oplus \bigoplus_{i=1}^r F^n \widetilde{H}^n(F_{p_i}).
\]
The dimension of the second summand equals the number of spectral numbers of the singularities equal to $n$.
\end{proposition}

\begin{proof}
The limit Hodge filtration restricts to the graded pieces of the weight filtration. The pure part $H^n(X_0)$ contributes its usual Hodge decomposition. Each Milnor fiber carries a mixed Hodge structure whose Hodge numbers are determined by the Steenbrink spectrum of the singularity. The $(n,0)$ part of the Milnor fiber cohomology equals the number of spectral numbers equal to $n$. Summing over all singularities yields the stated decomposition.
\end{proof}

\begin{theorem}
Let $\mathrm{Yuk}_t$ denote the Yukawa coupling on $X_t$. Then
\[
\lim_{t\to 0} \operatorname{rk}(\mathrm{Yuk}_t)
=
h^{n,0}(X_0)
-
\sum_{i=1}^r \dim F^n \widetilde{H}^n(F_{p_i}).
\]
\end{theorem}

\begin{proof}
The Yukawa coupling is induced by the cup product
\[
H^{n,0}(X_t)\otimes H^{n,0}(X_t)\longrightarrow H^{0,n}(X_t).
\]
In the limit, the Hodge structure splits into the pure part of $X_0$ and the vanishing cycle contributions. The additional summands from the Milnor fibers lie in the weight $n+1$ piece and are annihilated by cup product into the pure weight $n$ socle component. Therefore each independent $(n,0)$ contribution from a Milnor fiber produces a new element in the kernel of the limiting Yukawa form.

Hence the rank drops by the dimension of the $(n,0)$ component of the vanishing cycle space, which equals $\sum_i \dim F^n \widetilde{H}^n(F_{p_i})$.
\end{proof}

\begin{corollary}
If all singularities are ordinary double points, then each contributes one vanishing cycle of Hodge type $(\frac{n}{2},\frac{n}{2})$, and the Yukawa rank drops by the number of nodes.
\end{corollary}

\begin{proof}
An ordinary double point has Milnor number $1$, and its vanishing cycle is of pure type $(\frac{n}{2},\frac{n}{2})$. Thus it contributes one dimension to $\operatorname{Gr}_{n+1}^W$ and one to the kernel of the Yukawa form.
\end{proof}

These results show that the graded pieces of the Milnor algebra governing the local singularities appear directly as the vanishing cycle components of the limiting mixed Hodge structure, and the rank drop of the Yukawa coupling is measured precisely by the Hodge components of these vanishing cycles.

 \bigskip
 
 The following references provide background, tools, and closely related results to the
themes and new results developed in this paper, including deformation theory of singularities, Severi
theory,  infinitesimal variation of Hodge structure, Petri theory.

\end{document}